\documentclass[twoside]{irmaems}
\usepackage{amssymb} %use \usepackage[psamsfonts]{amssymb} if your fonts are from Y&Y/Blue Sky Research
\usepackage{amsmath} %use \usepackage{amsmath,cmmib57} if your fonts are from Y&Y/Blue Sky Research
\usepackage{latexsym}

\renewcommand\theequation{\arabic{section}.\arabic{equation}}
\newcommand\bc{\mathcal{B}}      
\newcommand\cc{\mathcal{C}}      
\newcommand\dc{\mathcal{D}}      
\newcommand\ec{\mathcal{E}}      
\newcommand\fc{\mathcal{F}}      
      
      \newcommand\hb{\mathbb{H}}

\newcommand\lc{\mathcal{L}}      
\newcommand\mc{\mathcal{M}}      
      \newcommand\nb{\mathbb{N}}

      \newcommand\qb{\mathbb{Q}}
\newcommand\rc{\mathcal{R}}      \newcommand\rb{\mathbb{R}}

\newcommand\uc{\mathcal{U}}

\theoremstyle{definition} %%% for statements in roman typeface

 \newtheorem{definition}{Definition}[section]

\theoremstyle{plain}      %%% for statements in italic typeface

 \newtheorem{proposition}[definition]{Proposition}
 \newtheorem{theorem}[definition]{Theorem}
 
 \newtheorem{lemma}[definition]{Lemma}

\def\opn#1#2{\def#1{\operatorname{#2}}}
\opn\ess{ess}
\opn\loc{loc}
\opn\Card{Card}
\opn\mod{mod}
\opn\vol{vol}
\opn\base{base}
\opn\diam{diam}
\opn\id{id}
\opn\dist{dist}
\newcommand{\stob}{\overset{\; \circ}{B}}

\begin{document}

\title{On densest packings of equal balls
of ~$\rb^{n}$~ and Marcinkiewicz spaces}

\author{Gilbert Muraz and Jean-Louis Verger-Gaugry}

\address{
Institut Fourier - CNRS UMR 5582, Universit\'e
Grenoble I,\\
BP 74 - Domaine Universitaire, \\
38402 - Saint Martin d'H\`eres, France\\
email:\,
\tt{jlverger@ujf-grenoble.fr}}

\maketitle

\vspace{1.5cm}

\begin{abstract} 
We investigate,
by `` \`a la Marcinkiewicz" techniques 
applied to the (asymptotic) density function,
how dense systems of equal spheres of $\rb^{n},
n \geq 1,$ can be partitioned at infinity in order to
allow the computation of their density as
a true limit and not a limsup.
The density of a packing of equal balls
is the norm 1
of the characteristic function of the
systems of balls in the 
sense of Marcinkiewicz.
Existence Theorems
for densest sphere packings and completely saturated
sphere packings of maximal density are given new direct 
proofs.
\end{abstract}



\newpage

\tableofcontents  %comment this command in if you want a table of contents

\section{Introduction}
\label{S1}

The existence of densest sphere packings
in $\rb^{n}, n \geq 2,$ asked the question 
to know how they could be constructed.
The problem of constructing very dense
sphere packings between the bounds 
of Kabatjanskii-Levenstein and
Minkowski-Hlawka type bounds (see Fig. 1
in \cite{murazvergergaugry1}) remains open
\cite{bezdek}
\cite{cassels} 
\cite{conwaysloane}
\cite{gruberlekkerkerker}
\cite{handbook}
\cite{rogers} 
\cite{zong}. 
There are two problems: the first one 
is the determination
of the supremum  $\delta_n$ over all 
possible densities, 
$\delta_n$ being called the packing constant, 
as a function of $n$ only (for $n=3$ see
Hales \cite{hales});
the second one consists in characterizing the
(local, global) 
configuration of balls in a densest sphere packing,
namely for which the density is $\delta_n$.

The notion of complete saturation was introduced
by Fejes-Toth, Kuperberg and Kuperberg 
\cite{fejestothkuperbergkuperberg}.
Section \ref{S2} gives new direct proofs of
the existence Theorems for completely 
saturated sphere packings (see Bowen \cite{bowen}
for a proof with $\rb^{n}$ and
$\hb^{n}$ as ambient spaces) of maximal density and 
densest sphere packings
in $\rb^{n}$. For this purpose new metrics 
are introduced (Subsection \ref{S2.1})
on the space of
uniformly discrete sets 
(space of equal sphere packings), and this
leads to a continuity Theorem for the density
function (Theorem \ref{extract}).

Let $\Lambda$ be a uniformly discrete
set of $\rb^{n}$ of constant 
$r > 0$, that is a discrete point set
for which $\|x - y \| \geq r$ for all $x, y \in \Lambda$,
with equality at least for one couple
of elements of $\Lambda$,
and consider the system of spheres (in fact {\it balls})
~$\bc(\Lambda) = \{
\lambda + B(0, \frac{r}{2}) \mid \lambda \in \Lambda \}$,
where $B(c,t)$ denotes the closed ball 
of center $c$ and radius $t$.
Let $B = B(0,1/2)$. The fact that the density 
$$
\delta(\bc(\Lambda)) :=
\limsup_{T \to +\infty}
\left[
\mbox{vol}
\bigl(
\bigl(
\bigcup_{\lambda \in \Lambda}
(\lambda + B(0,r/2))
\bigr)
\bigcap
B(0,T)
\bigr)
/\mbox{vol}
(B(0,T))
\right]
$$
of
$\bc(\Lambda)$ is equal to the 
norm (``norm 1") of Marcinkiewicz
of the characteristic function 
$\chi_{\bc(\Lambda)}$ of 
$\bc(\Lambda)$
\cite{bertrandias2}
\cite{phamphuhien}
\cite{marcinkiewicz},
namely
\begin{equation}
\label{egalmain}
\delta(\bc(\Lambda)) = \|\chi_{\bc(\Lambda)}\|_{1},
\end{equation}
where, for all $p \in \rb^{+*}$
and all $f \in \lc^{p}_{loc}$
with
$\lc^{p}_{loc}$ the
space of
complex-valued functions
$f$ defined on $\rb^{n}$
whose $p$-th
power of the absolute value
$|f|^{p}$ is integrable over
any bounded measurable subset
of
$\rb^{n}$ for the Lebesgue measure,
\begin{equation}
\label{normep2b}
\|f\|_{p} := \limsup_{t \to +\infty}
|f|_{p,t},
\end{equation}
with
\begin{equation}
\label{moyennep}
|f|_{p,t} := \left(
\frac{1}{\mbox{vol}(t B)}
\int_{t B} \, |f(x)|^{p} dx
\right)^{1/p}, \qquad f \in
\lc^{p}_{loc},
\end{equation}
asks the following question: 
what can tell the theory of Marcinkiewicz spaces to
the problem of constructing 
very dense sphere packings ? Obviously
the problem of the determination
of the packing constant or more generally of
the density is associated with
the quotient space  $\lc^{p}_{loc}/ \rc$ where
$\rc$ is the 
Marcinkiewicz equivalence relation 
(Section \ref{S3}): the density
function is a class function, that is 
is well defined on the Marcinkiewicz space
$\mc^p$ with $p = 1$. 
For instance any finite cluster of spheres has the same
density, equal to  zero, as the empty packing 
(no sphere); the Marcinkiewicz class of
the empty sphere packing being much larger than the set
of finite clusters of spheres.
Then it suffices to understand the construction of
one peculiar sphere packing per Marcinkiewicz class.
It is the object of this note to precise 
the geometrical 
constraints given by such a construction.

Since any non-singular affine transformation $T$
on a system of balls $\bc(\Lambda)$ leaves its density
invariant (Theorem 1.7 in \cite{rogers}), 
namely 
\begin{equation}
\label{egalaff}
\delta(\bc(\Lambda)) = 
\delta(T(\bc(\Lambda))),
\end{equation}
we will only consider
packings of spheres of common radius $1/2$ 
in the sequel. It amounts to consider the space
$\uc\dc$ of uniformly discrete subsets of
$\rb^n$ of constant 1. Its elements will be called
$\uc\dc$-sets. Denote by $\overline{f}$ 
the class in $\mc^p = 
\lc^{p}_{loc}/ \rc$ of $f \in 
\lc^{p}_{loc}$, where
$\lc^{p}_{loc}$ is endowed with the
$\mc^{p}$-topology (Section \ref{S3}),
and by
$$\begin{array}{rccrc}
\nu :&
\uc\dc ~\to~ \lc^{1}_{loc},
& \quad \mbox{resp.}&
\overline{\nu} :&
\uc\dc ~\to~ \mc^{1}\\
&
\Lambda ~\to~ \chi_{\bc(\Lambda)}&
&
&\Lambda ~\to~ \overline{\chi_{\bc(\Lambda)}}
\end{array}
$$
\noindent
the (set-) embedding of ~$\uc\dc$~ in
~$\lc^{1}_{loc}$, resp. in ~$\mc^{1}$.

\begin{theorem}
\label{close}
The image $\nu(\uc\dc)$ in $\lc_{loc}^{1}
\cap \lc^{\infty}$,
resp. $\overline{\nu}(\uc\dc)$
in $\mc^{1}$,
is closed.
\end{theorem}

Theorem \ref{close} is a reformulation of the 
following more accurate theorem, since
$\mc^{p}$ is complete \cite{bertrandias1}
\cite{bertrandias2}.
For $0 \leq \lambda \leq \mu$
denote
$$\cc(\lambda, \mu) := 
\{x \in \rb^{n} \mid
\lambda \leq \|x\| \leq \mu \}$$
the closed annular region of space 
between the spheres centered at the origin
of respective radii $\lambda$ and $\mu$.

\begin{theorem}
\label{couronne}
Let $(\Lambda_{m})_{m \geq 1}$ be
a sequence of ~$\uc\dc$-sets
such that
the sequence
$(\chi_{\bc(\Lambda_{m})})_{m \geq 1}$
is a Cauchy sequence
for the pseudo-metric 
$\|\cdot\|_1$
on
~$\lc_{loc}^{1} \cap \lc^{\infty}$.
Then, 
there exist
\begin{itemize}
\item[(i)] a strictly
increasing sequence of
positive integers
$(m_{i})_{i \geq 1}$,
\item[(ii)]
a strictly increasing
sequence
of real numbers
$(\lambda_{i})_{i \geq 1}$~
with ~$\lambda_{i} \geq 1$~
and ~$\lambda_{i+1} > 2 \lambda_{i}$,
\end{itemize}
such that, with
\begin{equation}
\label{limud}
\Lambda =
\bigcup_{i \geq 1} \,
\Lambda_{m_{i}} \cap
\, \cc(\lambda_{i}+1/2,
\lambda_{i+1}-1/2),
\end{equation}
the two 
functions
$$\chi_{\bc(\Lambda)}~ \mbox{and} ~ 
\lim_{i \to +\infty} \, \chi_{\bc(\Lambda_{m_i})}$$
are $\mc^{1}$-equivalent.
As a consequence
\begin{equation}
\label{limdens}
\delta(\bc(\Lambda)) ~=~
\lim_{i \to +\infty}
\delta(\bc(\Lambda_{m_{i}})).
\end{equation}
%since
%$$| \delta(\bc(\Lambda_{m_{i}})) -
%\delta(\bc(\Lambda)) | ~\leq~
%\limsup_{t \to +\infty} \left(
%\frac{1}{\mbox{vol}(tB)}
%\int_{tB} |\chi_{\bc(\Lambda_{m_{i}})}(x)
%- \chi_{\bc(\Lambda)}(x)|^{p}
%\right)^{1/p} ~=~ o(1).
%$$
\end{theorem}

The situation is the following
for a (densest) sphere packing $\bc(\Lambda)$ of
$\rb^n$ for which $\delta(\bc(\Lambda)) = \delta_n$ :
\begin{itemize}
\item[$\ast$] either it cannot be 
reached by a sequence
of sphere packings 
such as in Theorem \ref{couronne}, 
in which case there
is an {\it isolation phenomenon},
\item[$\ast$] or there 
exists at least one sequence 
of sphere packings such as in 
Theorem \ref{couronne}, and 
it
is Marcinkiewicz - equivalent to a
sphere packing
having the asymptotic
annular structure given 
by Theorem \ref{couronne}, where the sequence
of thicknesses of the annular portions 
exhibit an exponential growth. 
\end{itemize}

The sharing of space in 
annular portions as given by 
Theorem \ref{couronne}
may allow constructions of very dense
packings of spheres layer-by-layer in each
portion independently, since
the intermediate
regions $\cc(\lambda_i - 1/2, \lambda_i + 1/2)$
are all of constant thickness 1 which is twice the 
common ball radius 1/2.
These intermediate 
regions do not contribute to the
density so that they can be filled up or not by spheres. 
However
the existence of
such unfilled spherical gaps are not
likely to provide completely 
saturated packings, at least for $n=2$ 
\cite{kuperbergkuperbergkuperberg}.

Note that the value $2$ which controls
the exponential sequence of radii $(\lambda_i)_i$
by $\lambda_{i+1} > 2 \lambda_i$ in
Theorem \ref{couronne} (ii) can be replaced by
any value $a > 1$. This is important for 
understanding constructions of sphere packings
iteratively on the dimension $n$: 
indeed, chosing $a > 1$ sufficiently small
brings the problem back to fill
up first one layer in a as dense as possible way, 
therefore in dimension $n-1$, then propagating 
towards the orthogonal direction exponentially.

The terminology
{\it density} is usual in the field of lattice
sphere packings, while the terminology
{\it asymptotic measure}, therefore
{\it asymptotic density}, is usual 
in Harmonic Analysis,
both meaning the same in the 
present context.

\vspace{0.5cm}

\section{Densest sphere packings and complete saturation}
\label{S2}

The set ~$SS$~ of
systems of equal spheres
of radius ~$1/2$~ and
the set ~$\uc\dc$~
are in one-to-one correspondence: 
$\Lambda =
(a_{i})_{i \in \nb} \in \uc\dc$~ is
the set of sphere centres of
~$\bc(\Lambda) =
\{a_{i} + B ~|~ i \in \nb\}
\in SS$. More conveniently we will use
the set $\uc\dc$ of point sets of $\rb^n$
instead of $SS$. The subset of $\uc\dc$ of
finite uniformly discrete sets of constant $1$ 
of $\rb^n$ is denoted by $\uc\dc_f$.

\subsection{A metric on $\uc\dc$
invariant by the rigid motions of
$\rb^{n}$}
\label{S2.1}

Denote by $O(n,\rb)$ the $n$-dimensional
orthogonal group of $n \times n$ matrices
$M$, i.e. such that $M^{-1}=\mbox{}^{t}M$.
A {\it rigid motion} (or
an {\it Euclidean displacement}) is
an ordered pair
~$(\rho, t)$~ with ~$\rho \in O(n, \rb)$~ and
~$t \in \rb^{n}$ \cite{charlap}. 
The composition of
two rigid motions is given by
~$(\rho, t)(\rho',t') = (\rho \rho', \rho(t') + t)$~
and the group of rigid motions 
is the split extension of ~$O(n, \rb)$~ by
~$\rb^{n}$ (as a semi-direct product). 
It is endowed with the 
usual topology. 
Theorem \ref{UDselection}, obtained as a generalization
of the Selection Theorem of Mahler
\cite{chabauty}
\cite{gruberlekkerkerker}
\cite{mahler}
\cite{martinet},
gives the existence of
a metric $d$ on $\uc\dc$
\cite{murazvergergaugry2} which extends the Hausdorff
metric on the subspace 
$\uc\dc_{f}$. The metric $d$ 
is not invariant by
translation. From it, adding to the construction
of $d$ some additional 
constraints so that 
it gains in invariant properties 
(Proposition \ref{UDstructure} iii) 
proved in
Section \ref{S5}),
a new metric $D$, invariant by translation
and by the group of rigid motions of
$\rb^n$ (Theorem \ref{metricuniv} 
proved in Section \ref{S6}), 
can be constructed, giving a new topology to
$\uc\dc$, suitable for studying the continuity 
of the density function 
(Theorem \ref{extract}). 

%way to the theorem \ref{metricuniv}, i.e. 
%~$d^{(r)}$~ and ~$D^{(r)}.$
%
%Restricted to ~$\uc\dc_{r} \setminus
%\{\emptyset\}$, the metric ~$D^{(r)}$~
%will be equivalent to the Hausdorff metric.
%The main difference between ~$d^{(r)}$~ and
%~$D^{(r)}$~ will be that ~$D^{(r)}$~ will be invariant by
%translation and ~$d^{(r)}$~ not. 

\begin{theorem}
\label{UDselection}
The set $\uc\dc$ can be endowed with a metric
$d$ such that the topological space 
$(\uc\dc, d)$ is compact and such that the Hausdorff
metric $\Delta$ on $\uc\dc_f$ is compatible
with the restriction of the topology of
$(\uc\dc, d)$ to $\uc\dc_f$.
\end{theorem}

\begin{proof}
Theorem 1.2 in \cite{murazvergergaugry2}.
\end{proof}

%The metric $d$ on $\uc\dc$ 
%can be constructed 
%in order to exhibit the invariance property
%given by iii) in Proposition \ref{USstructure}.

\begin{proposition}
\label{UDstructure}
There exists a metric ~$d$~ on
~$\uc\dc$~ such that: 
\begin{itemize}
\item[i)]
the space $(\uc\dc, d)$
is compact,
\item[ii)] the Hausdorff metric
on $\uc\dc_f$ is compatible 
with the restriction
of the topology of
$(\uc\dc, d)$ to $\uc\dc_f$,
\item[iii)]
$d(\Lambda, \Lambda') =
d(\rho(\Lambda), \rho(\Lambda'))$
for all $\rho \in
O(n, \rb)$ and $\Lambda, \Lambda' \in
\uc\dc$.
\end{itemize}
\end{proposition}

Since the density of a sphere
packing is left invariant by any
non-singular affine transformation
(\eqref{egalaff}; Theorem 1.7 in Rogers \cite{rogers}),
it is natural to construct metrics 
on $\uc\dc$ which are at least 
invariant by the translations  
and by the orthogonal group of
$\rb^n$. Such a metric is given by 
the following theorem. 

\begin{theorem}
\label{metricuniv}
There exists a metric ~$D$~ on ~$\uc\dc$~
such that: 
\begin{itemize}
\item[i)] 
~$D(\Lambda_{1}, \Lambda_{2}) =
$
$D(\rho(\Lambda_{1}) $
$+ t, \rho(\Lambda_{2}) + t)$~ 
for all ~$t \in \rb^{n}, \rho \in O(n, \rb^{n})$~ and
all ~$\Lambda_{1}, \Lambda_{2} \in \uc\dc$,
\item[ii)]
~the space ~$(\uc\dc, D)$~
is complete and locally compact,
\item[iii)]
~(pointwise pairing property)
~ for all
non-empty
~$\Lambda, \Lambda' \in \uc\dc$~
such that ~$D(\Lambda, \Lambda') < \epsilon$,
each point ~$\lambda \in \Lambda$~ is associated with
a unique point ~$\lambda' \in \Lambda'$~
such that
~$\|\lambda - \lambda'\| < \epsilon/2$,
\item[iv)]
~the action of the group of rigid motions
~$O(n, \rb) \ltimes \rb^{n}$~ on
~$(\uc\dc, D):
((\rho, t),\Lambda) \to (\rho, t) \cdot \Lambda =
\rho(\Lambda) + t$~ is such that
its subgroup of translations ~$\rb^{n}$~ acts 
continuously
on ~$\uc\dc$. 
\end{itemize}
\end{theorem}

%Let us turn to packings of balls 
%\cite{conwaysloane} \cite{gruberlekkerkerker}
%\cite{handbook}
%\cite{oesterle} \cite{rogers} 
%\cite{zong} and
%let us denote by ~$B(c,r)$~ the 
%(generic) closed ball of ~$\rb^{n}$~
%of centre
%~$c$~ and of radius ~$r > 0$.
%For all ~$\Lambda \in \uc\dc$~ having at least two 
%distinct points, let us denote by
%~$\bc(\Lambda)$~ the system of balls ~$\Lambda +
%B(0,m(\Lambda)/2)$. 
%This system of balls forms a packing whose
%density is defined by ~$\delta(\bc(\Lambda)) =
%\limsup_{R \to +\infty} \mbox{vol}(\bc(\Lambda) \cap B(0,R))
%/\mbox{vol}(B(0,R))$. 
%The supremum ~$\delta = \sup_{\Lambda \in \uc\dc}
%\delta(\bc(\Lambda))$~ will 
%be called {\it the packing constant}.
%It depends only upon ~$n$~ and 
%is linked to the maximal size of
%holes in densest 
%packings \cite{murazvergergaugry1}
%\cite{murazvergergaugry2}.
%As a consequence of the theorems 
%\ref{UDstructure} and \ref{metricuniv},
%we will prove the following results 
%in section \ref{S3}.

\vspace{0.5cm}

\subsection{Existence Theorems}
\label{S2.2}

The two following Theorems rely upon the continuity
of the density function $\|\ \cdot \|_1 \circ \nu$ on the 
space $(\uc\dc, D)$ (Theorem \ref{metricuniv} and
Theorem \ref{extract}).

\begin{theorem}
\label{maxsphere}
There exists an element 
$\Lambda \in \uc\dc$ such that
the following equality holds:
\begin{equation}
\label{maximu}
\delta(\bc(\Lambda)) ~=~ \delta_n.
\end{equation}
\end{theorem}

\begin{proof}
See Groemer \cite{groemer} and Section \ref{S7}.
\end{proof}

%By restricting the density function
%~$\Lambda \to \delta(\bc(\Lambda))$~ to
%the subspace of lattices of ~$\rb^{n}$, 
%we obtain
%the existence of extreme lattices,
%without
%invoking any 
%theory of reduction for lattices
%\cite{martinet} \cite{oesterle}. 

%\begin{coro}
%\label{reseauextreme}
%Let ~$\lc_{n}$~ be the locally compact topological space
%of lattices of ~$\rb^{n}$.
%Let ~$\delta_{L} := 
%\sup_{\Lambda \in \lc_{n}} \delta(\bc(\Lambda))$.
%Then, there exists  ~$\Lambda \in \lc_{n}$~ such that
%the following equality holds:
%~$\delta(\bc(\Lambda)) ~=~ \delta_{L}$.
%\end{coro}

We will
say that
~$\Lambda \in \uc\dc$~
is {\em saturated},
or {\em maximal},
if it is
impossible to
add a replica of the ball $B$
(a ball of radius
$1/2$) to ~$\bc(\Lambda)$~ 
without destroying
the fact that it is a packing of balls,
i.e. without creating an overlap of balls.
The set ~$SS$~ of
systems of balls of radius ~$1/2$,
is partially ordered by the relation
~$\prec$~ defined by
~$\Lambda_{1}, \Lambda_{2} \in \uc\dc, ~~~~
\bc(\Lambda_{1}) \prec
\bc(\Lambda_{2}) 
~~\Longleftrightarrow~~
\Lambda_{1} \subset \Lambda_{2}$.
By Zorn's lemma, maximal packings of balls exist.
The saturation operation
of a packing of balls
consists in adding balls to obtain
a maximal packing of balls.
It is fairly arbitrary and 
may be finite or
infinite.
More generally
\cite{fejestothkuperbergkuperberg},
~$\bc(\Lambda)$~ is said to be
{\it ~$m$-saturated}
if no finite subsystem of
~$m-1$~ balls of it can be replaced with
~$m$~ replicas of the ball ~$B(0,r/2)$.
The notion of $m$-saturation was introduced
by Fejes-Toth, Kuperberg and Kuperberg
\cite{fejestothkuperbergkuperberg}.
Obviously,
~$1$-saturation means saturation, and
~$m$-saturation implies ~$(m-1)$-saturation.
It is not because a packing of balls is
saturated, or
~$m$-saturated, that
its density is
equal to ~$\delta_n$. 
The packing ~$\bc(\Lambda)$~ is 
{\it completely saturated}
if it is ~$m$-saturated for every ~$m \geq 1$. 
Complete saturation is a sharper version
of maximum density \cite{kuperberg}. 

\begin{theorem}
\label{complsat}
Every ball in ~$\rb^{n}$~ admits a completely saturated
packing with replicas of the ball, 
whose density is equal to the packing constant
~$\delta_n$. 
\end{theorem}

\begin{proof}
Theorem 1.1 in
\cite{fejestothkuperbergkuperberg}.
See also Bowen 
\cite{bowen}. A direct proof is given in 
Section \ref{S7}, where we prove that
there always exists a completely saturated
sphere packing in the Marcinkiewicz class of
a densest sphere packing. 
\end{proof}

\vspace{0.5cm}

\section{Marcinkiewicz spaces and norms} 
\label{S3}

Let $p \in \rb^{+*}$.
The Marcinkiewicz ~$p$-th space
~$\mc^{p}$~ is
the quotient space of the
subspace
~$\{ f \in
\lc^{p}_{loc} ~|~ \|f\|_{p} < +\infty\}$~
of ~$\lc^{p}_{loc}$~
by the equivalence relation
$\rc$
which identifies ~$f$~ and ~$g$~ as soon as
~$\|f - g\|_{p} = 0$
(Marcinkiewicz \cite{marcinkiewicz},
Bertrandias \cite{bertrandias1},
Vo Khac \cite{vokhac}):
\begin{equation}
\label{defimarcin}
\mc^{p} := \{ \overline{f} ~|~ f \in
\lc^{p}_{loc}, \|f\|_{p} < +\infty\}.
\end{equation}
This equivalence relation
is called Marcinkiewicz
equivalence relation.
It is usual to introduce, with
$|f|_{p,t}$ given by 
\eqref{moyennep},
the two semi-norms
$$
\|f\|_{p} := \limsup_{t \to +\infty}
|f|_{p,t}
$$
and
$$
|\|f\||_{p} := \sup_{t > 0} |f|_{p,t}
$$
on $\lc^{p}_{loc}$. 
The vector 
space ~$\mc^{p}$~ is then normed
with 
~$\|\overline{f}\|_{p} =
\|f\|_{p}$. 

\begin{theorem}
\label{marcincompl}
The space $\mc^{p}$ is
complete.
\end{theorem}

\begin{proof}
Marcinkiewicz
\cite{marcinkiewicz},
\cite{bertrandias1},
\cite{vokhac}. 
\end{proof}

We call
~$\mc^{p}$-topology the topology
induced by this norm on 
~$\mc^{p}$~ or on 
~$\lc^{p}_{loc}$~ itself.
Both spaces will be endowed with
this topology. 
%We identify functions
%with their equivalence classes.

Following Bertrandias 
\cite{bertrandias1} we 
say that a function
$f \in \lc_{loc}^{p}$ is
$\mc^{p}$-regular if
$$\lim_{t \to +\infty}
\frac{1}{\mbox{vol(t B)}} \, 
\int_{t B \setminus
(t-l) B}
|f(x)|^{p} dx = 0
\qquad
\mbox{for all real number ~}
l.$$
Since all functions
$f \in \lc_{loc}^{p}$ such that
$\|f\|_{p} = 0$ are 
$\mc^{p}$-regular, we consider
classes of $\mc^{p}$-regular
functions of $\lc_{loc}^{p}$ 
modulo the
Marcinkiewicz equivalence relation.
We call 
$\mc_{r}^{p}$ the set 
of classes of Marcinkiewicz
equivalent $\mc^{p}$-regular functions.
\begin{proposition}
\label{mpregular}
The set
$\mc_{r}^{p}$ is a complete 
vector subspace of ~$\mc^{p}$.
\end{proposition}

\begin{proof}
\cite{bertrandias1}.
\end{proof}

\vspace{0.5cm}

\section{Proof of Theorem \ref{couronne}} 
\label{S4}

Theorem \ref{couronne} is the $n$-dimensional
version of the remark of Marcinkiewicz
\cite{marcinkiewicz} in the case $p = 1$.
We prove a theorem 
slightly stronger than Theorem
\ref{couronne}
(Theorem \ref{cauchy}), 
by making the assumption 
in Lemma \ref{zerof}
and in Theorem \ref{cauchy} that $p$
is $\geq 1$ in full generality. 

\begin{lemma}
\label{zerof}
Let $p \geq 1$.
Let $(\lambda_{i})_{i \geq 1}$ 
be a
sequence of real numbers such that
$$\lambda_{i} \geq 1, ~~\lambda_{i+1}
> 2 \lambda_{i}, \qquad \qquad i \geq 1.$$
Let 
$\cc_i := 
\cc 
\bigl( 
\lambda_{i}+1/2,
\lambda_{i+1}-1/2
\bigr)$.
Then for all bounded function
$f \in \lc_{loc}^{p}$~ such that
$f_{|_{\cc_i}}    
\equiv 0$
for all $i \geq 1$,
we have
$$\|f\|_{p} = 0$$
\end{lemma}

\begin{proof}
Immediate.
\end{proof}

Lemma \ref{zerof} is the special case
of $\mc^{p}$-regularity applied to 
the characteristic functions
of systems of spheres which eventually
lie within the 
spherical intermediate regions
$\cc_i$. It proves that such spheres 
do not contribute
to the density anyway.

\begin{theorem}
\label{cauchy}
Let $p \geq 1$.
Let $(\Lambda_{m})_{m \geq 1}$ be
a sequence of ~$\uc\dc$-sets
such that
the sequence
$(\chi_{\bc(\Lambda_{m})})_{m \geq 1}$
is a Cauchy sequence
for the pseudo-metric 
$\|\cdot\|_p$
on
$\lc_{loc}^{p} \cap \lc^{\infty}$.
Then, 
there exist
\begin{itemize}
\item[(i)] a strictly
increasing sequence of
positive integers
$(m_{i})_{i \geq 1}$,
\item[(ii)]
a strictly increasing
sequence
of real numbers
$(\lambda_{i})_{i \geq 1}$~
with ~$\lambda_{i} \geq 1$~
and ~$\lambda_{i+1} > 2 \lambda_{i}$,
\end{itemize}
such that, with
\begin{equation}
\label{limud}
\Lambda =
\bigcup_{i \geq 1} \,
\Lambda_{m_{i}} \cap
\, \cc(\lambda_{i}+1/2,
\lambda_{i+1}-1/2),
\end{equation}
the two functions
\begin{equation}
\label{limudchi}
\chi_{\bc(\Lambda)} \qquad 
\mbox{and} \qquad
\lim_{i \to +\infty} \, \chi_{\bc(\Lambda_{m_i})}
%\qquad \quad (\mbox{pointwise and for}~ \|\cdot\|_p),
\end{equation}
are $\mc^{p}$-equivalent.
\end{theorem}

\begin{proof} 

Since
the sequence
$(\chi_{\bc(\Lambda_{m})})$ 
is a Cauchy sequence, 
let us chose a subsequence
of $\uc\dc$-sets 
$(\Lambda_{m_{i}})_{i \geq 1}$
which satisfies
$$ \|\chi_{\bc(\Lambda_{m_{i}})} 
- \chi_{\bc(\Lambda_{m_{i+1}})}\|_{p} 
\leq 2^{-(i+1)}\, .$$ 
Then, denoting
$$R_{\lambda}(f) := 
\sup_{\lambda +1/2 \leq t < +\infty}
\left(
\frac{1}{\mbox{vol}(t B)}
\int_{t B} | f(x) |^{p} dx
\right)^{1/p},
\qquad f \in \lc_{loc}^{p},$$
let us chose
a sequence of real numbers
$(\lambda_{i})_{i \geq 1}$
for which $\lambda_{i} \geq 1,
\lambda_{i+1} > 2 \lambda_{i}$, and
such that
$$R_{\lambda_{i}}(\chi_{\bc(\Lambda_{m_{i}})}
-
\chi_{\bc(\Lambda_{m_{i+1}})}) 
\leq 2^{-i} .$$
Let us define the function
$$H(x) :=
\left\{
\begin{array}{lll}
\chi_{\bc(\Lambda_{m_{i}})}(x) 
& \mbox{if}
& \lambda_{i} + 1/2 \leq \|x\| \leq \lambda_{i+1} - 1/2
~~~~(i=1, 2, \dots),\\
0 & \mbox{if} & 
\lambda_{i} - 1/2 < \|x\| < \lambda_{i} + 1/2
~~~~(i=1, 2, \dots),\\
0 & \mbox{if} &
\|x\| \leq \lambda_{1} - 1/2.
\end{array}
\right.
$$
The function $H(x)$ 
is exactly the characteristic function
of $\bc(\Lambda)$ on
$J := \bigcup_{j=1}^{+\infty}
\, \cc_{j}$ the portion of space
occupied by the closed annuli $\cc_j$. 
Let us prove that the function $H(x)$ satisfies:
\begin{equation}
\label{Hlimite}
\lim_{i \to +\infty} \| H -
\chi_{\bc(\Lambda_{m_{i}})}\|_{p} = 0.
\end{equation} 
Let us fix $i$ and take $t$ and $k$
such that 
\begin{equation}
\label{tki}
\lambda_{k} + 1/2
\leq 2 t \leq \lambda_{k+1}-1/2
\end{equation}
holds with $k \geq i + 1$.
Then
$$\int_{t B} |H(x) - 
\chi_{\bc(\Lambda_{m_{i}})}(x)|^{p} dx
~=~ \int_{t B \cap J}
|H(x) - \chi_{\bc(\Lambda_{m_{i}})}(x)|^{p} dx$$
$$
+ \int_{t B \cap (\rb^{n} \setminus
J)}
\chi_{\bc(\Lambda_{m_{i}})}(x)^{p} dx.
$$
By Lemma \ref{zerof},
$$\|\chi_{\bc(\Lambda_{m_{i}})} \cap
\chi_{\rb^{n} \setminus
J}\|_{p} ~~=~~0.$$
Hence, we have just to consider 
the portion of space occupied by the spheres 
$\bc(\Lambda_{m_i})$ in $t B \cap J$. 
We have
$$\int_{tB \cap J} |H(x) - 
\chi_{\bc(\Lambda_{m_{i}})}(x)|^{p} dx
~=~ \sum_{\nu = 1}^{i}
\int_{t B \cap \cc_{\nu}} 
|H(x) - \chi_{\bc(\Lambda_{m_{i}})}(x)|^{p} dx$$
$$  +
\sum_{\nu = i+1}^{k-1}
\int_{t B \cap \cc_{\nu}}
|H(x) - \chi_{\bc(\Lambda_{m_{i}})}(x)|^{p} dx
+\int_{tB \cap \cc_{k}}
|H(x) - \chi_{\bc(\Lambda_{m_{i}})}(x)|^{p} dx
$$
$$
= A + E + C.
$$
Let us now transform the sum ~$A$:
$$\sum_{\nu = 1}^{i}
\int_{tB \cap \cc_{\nu}} 
|H(x) - \chi_{\bc(\Lambda_{m_{i}})}(x)|^{p} dx
=
\sum_{\nu = 1}^{i}
\int_{tB \cap \cc_{\nu}}
|\chi_{\bc(\Lambda_{m_{\nu}})}(x) -
\chi_{\bc(\Lambda_{m_{i}})}(x)|^{p} dx.
$$ 
But,
for all ~$\nu \in \{1, 2, \dots, i-1\}$,
$$\left(
\int_{tB \cap \cc_{\nu}}
|\chi_{\bc(\Lambda_{m_{\nu}})}(x) -
\chi_{\bc(\Lambda_{m_{i}})}(x)|^{p} dx
\right)^{1/p}$$
$$
\leq
\sum_{\omega = \nu}^{i-1}
\left(
\int_{tB \cap \cc_{\nu}}
|\chi_{\bc(\Lambda_{m_{\omega}})}(x)
-
\chi_{\bc(\Lambda_{m_{\omega+1}})}(x)|^{p}
dx
\right)^{1/p}
$$
$$
\leq \sum_{\omega = \nu}^{i-1}
\left(
\mbox{vol}((\lambda_{\omega}+1/2)B)\right)^{1/p}
\,
R_{\lambda_{\omega}}(\chi_{\bc(\Lambda_{m_{\omega}})}
-
\chi_{\bc(\Lambda_{m_{\omega+1}})})
$$
$$
\leq \sum_{\omega = \nu}^{i-1}
\left(
\mbox{vol}((\lambda_{i}+1/2)B)\right)^{1/p}
\, 2^{-\omega} \leq 
\left(\mbox{vol}((\lambda_{i}+1/2)B)\right)^{1/p}.
$$
Hence
\begin{equation}
\label{A}
A \leq i \, \mbox{vol}((\lambda_{i}+1/2)B).
\end{equation}
Let us transform the sum E:
$$\sum_{\nu = i+1}^{k-1}
\int_{t B \cap \cc_{\nu}}
|H(x) - \chi_{\bc(\Lambda_{m_{i}})}(x)|^{p} dx 
=
\sum_{\nu = i+1}^{k-1}
\int_{t B \cap \cc_{\nu}}
|\chi_{\bc(\Lambda_{m_{\nu}})}(x)
-
\chi_{\bc(\Lambda_{m_{i}})}(x)|^{p} dx.$$
But, for
all ~$\nu \in \{i+1, i+2, \dots, k-1\}$,
$$
\left(
\int_{t B \cap \cc_{\nu}}
|\chi_{\bc(\Lambda_{m_{\nu}})}(x)
-
\chi_{\bc(\Lambda_{m_{i}})}(x)|^{p} dx
\right)^{1/p}
$$
$$\leq
\sum_{\omega=i}^{\nu-1}
\left(
\int_{t B \cap \cc_{\nu}}
|\chi_{\bc(\Lambda_{m_{\omega}})}(x)
-
\chi_{\bc(\Lambda_{m_{\omega + 1}})}(x)|^{p} dx
\right)^{1/p}
$$
$$
\leq
\sum_{\omega=i}^{\nu-1}
\left(
\mbox{vol}((\lambda_{\omega} + 1/2) B)
\right)^{1/p}
\,
R_{\lambda_{\omega}}(
\chi_{\bc(\Lambda_{m_{\omega}})}
-
\chi_{\bc(\Lambda_{m_{\omega + 1}})})
$$
$$\leq
\sum_{\omega=i}^{\nu-1}
\left(
\mbox{vol}((\lambda_{\omega} + 1/2) B)
\right)^{1/p}
\, 2^{-\omega}
\leq
\mbox{vol}((\lambda_{\nu} + 1/2) B)^{1/p}
\,
2^{-i+1}.$$
Hence,
$$E \leq 
2^{-(i-1)p} \left(
\mbox{vol}((\lambda_{i+1} + 1/2) B)
+
\mbox{vol}((\lambda_{i+2} + 1/2) B)
+
\dots \hspace{1cm} \mbox{}\right.$$ 
$$
\left.
\mbox{} \hspace{8cm}
+
\mbox{vol}((\lambda_{k-1} + 1/2) B)
\right)
$$
$$
\leq 
2^{-(i-1)p} 
\, \mbox{vol}((\lambda_{k} + 1/2) B)
\left(
\frac{1}{2^{n}} +
\frac{1}{2^{2 n}} +
\dots 
\right) 
$$
\begin{equation}
\label{E}
\leq 2^{-(i-1)p}
\, \mbox{vol}(2 t B) 
= 2^{-(i-1)p +n}
\, \mbox{vol}(t B).
\end{equation}
Let us transform the sum C:
\begin{equation}
\label{C}
C \leq
2^{-(i-1)p +n}
\, \mbox{vol}(t B).
\end{equation}
From
\eqref{A}, 
\eqref{E} and \eqref{C} we deduce
$$
\left(
\frac{1}{\mbox{vol}(tB)} \,
\int_{t B \cap J} |H(x) - 
\chi_{\bc(\Lambda_{m_{i}})}(x)|^{p} dx
\right)^{1/p}
$$
$$\leq \left[
\frac{i \, \mbox{vol}((\lambda_{i}+1/2)B)}
{\mbox{vol}(t B)}
+ 2^{-(i-1)p + n + 1}
\right]^{1/p}. 
$$
Using \eqref{tki} we deduce, for a certain constant $c > 0$,
$$\| H -
\chi_{\bc(\Lambda_{m_{i}})}\|_{p}
\leq 
c \, 2^{-(i-1)}.$$
This implies \eqref{Hlimite}.
Now, if $m_{i} \leq q < m_{i+1}$, 
$i \geq 1$,
we have
$$\| H -
\chi_{\bc(\Lambda_{q})}\|_{p}
\leq 
\| \chi_{\bc(\Lambda_{m_{i}})}
-
\chi_{\bc(\Lambda_{q})}\|_{p}
+
\| H -
 \chi_{\bc(\Lambda_{m_{i}})}\|_{p}
= o(1) + o(1) = o(1)$$
when $i$ tends to $+\infty$.
The proof of the 
$\mc^{p}$-equivalence 
\eqref{limudchi} between
$H$ and $\lim_{i \to +\infty} 
\chi_{\bc(\Lambda_i)}$
is now complete. 

The thickness of the empty
annular intermediate regions
$\cc_{i}$
is equal to 1: it 
ensures that the limit 
point set $\Lambda$
is uniformly discrete 
of constant $1$.
\end{proof}

\vspace{0.5cm}

\section{Proof of Proposition \ref{UDstructure}}
\label{S5}

The metric $d$ on $\uc\dc$ was constructed
in
\cite{murazvergergaugry2}, \S 3.2.1,
as a kind of 
counting system  
normalized by a suitable
distance function. 
In order to make explicit 
the statement iii) of Proposition 
\ref{UDstructure}, we
recall the construction of $d$, adding the
ingredient \eqref{sondeii} in 
order to obtain the claim. The metric $d$ 
is given in Lemma \ref{l2.2}.

For all $\Lambda \in \uc\dc$, 
we denote by $\Lambda_{i}$ its $i$-th element.
Let 
$$\ec = \{ (D,E) \mid D~ \mbox{countable point set in}
~\rb^{n}, E~ \mbox{countable point set in}
~(0 , 1/2)\}$$
and
~$f : \rb^{n} \to [0 , 1 ]$ a continuous function
with compact support in $B(0, 1)$
which satisfies:
\begin{equation}
\label{sondei}
f(0) = 1,
\end{equation}
\begin{equation}
\label{sondeii}
f(\rho(t)) =
f(t) \qquad \mbox{for all}~ t \in \rb^{n}~ \mbox{and
all}~ \rho \in O(n, \rb),
\end{equation}
\begin{equation}
\label{sondeiii}
f(t) \leq \frac{1/2 + \| \lambda - t/2 \|}
{1/2+\| \lambda \|} \qquad \mbox{for all} ~t \in B(0,1)~ 
\mbox{and} 
~\lambda \in \rb^{n}.
\end{equation}

It is 
remarkable that
the topology of $(\uc\dc, d)$
does not depend upon $f$ once 
\eqref{sondei} and \eqref{sondeiii} 
are simultaneously satisfied
(\cite{murazvergergaugry2} 
Proposition 3.5 and \S 3.3).
Therefore adding \eqref{sondeii} 
does not change
the topology of $(\uc\dc, d)$ 
but only the invariance
properties of the metric $d$.

For $f$ for
instance, let us take
~$f(t) = 1 - 2 \|t\|$~ for
~$t \in B(0,1/2)$~ and
~$f(t)= 0$~ elsewhere.

With each element $(D,E) \in \ec$ 
and origin
$\alpha$ of  
~$\rb^{n}$~ we 
associate a real-valued function 
~$d_{\alpha,(D,E)}$~ on 
~$\uc\dc \times \uc\dc$~
in the following way 
(denoting by
~$\stob(c, v)$~ 
the interior of the closed
ball ~$B(c,v)$~ of centre ~$c$~ and
radius ~$v > 0$).
Let $\bc_{(D,E)} = 
\{ \bc_{m} \}$ denote 
the countable set of all possible finite 
collections 
$$\bc_{m} 
=\bigl\{ \stob(c_{1}^{(m)}, 
\epsilon_{1}^{(m)}), 
\stob(c_{2}^{(m)}, 
\epsilon_{2}^{(m)}),
\ldots, 
\stob(c_{i_{m}}^{(m)}, 
\epsilon_{i_{m}}^{(m)}) 
\bigr\}$$
%~(with ~$i_{m}$~ the number of elements
%~$\# \bc_{m}$~ of ~$\bc_{m}$) 
of open balls
such that ~$c_{q}^{(m)} \in D$ and 
~$\epsilon_{q}^{(m)} \in
E$~  
for all ~$q \in \{1, 2, \ldots, i_{m}\}$, 
and such that 
for all $m$ and any two distinct  
balls in $\bc_{m}^{(r)}$ of 
respective centers $c_{q}^{(m)}$ and 
~$c_{k}^{(m)}$, we have
$$\| c_{q}^{(m)} - c_{k}^{(m)} \| \geq 1.$$ 
Then we
define the following function, 
with ~$\Lambda, \Lambda' \in \uc\dc$,\\ 

$
d_{\alpha,(D,E)}(\Lambda,\Lambda') :=
$
\begin{equation}
\label{ad1}
\sup_{ 
\bc_{m} \in \bc_{(D,E)}} 
\frac{\left| \phi_{\bc_{m}}(\Lambda)- 
\phi_{\bc_{m}}(\Lambda') \right|}
{(1/2 + \|\alpha\|
+\|\alpha - c_{1}^{(m)} \| + \|\alpha - c_{2}^{(m)}\| + \dots +
\|\alpha - c_{i_{m}}^{(m)}\|)}
\end{equation}
where the function
~$\phi_{\bc_{m}}$~ is given by
$$\phi_{\bc_{m}}(\Lambda) := 
\sum_{\stob(c,\epsilon) 
\in \bc_{m}} 
\sum_{i} \epsilon f\left(\frac{\Lambda_{i}-c}{\epsilon}\right),$$ 
putting 
$\phi_{\bc_{m}}(\emptyset) = 0$~ for all
~$\bc_{m} \in \bc_{(D,E)}$~ and
all ~$(D,E) \in \ec$ by convention.

\begin{lemma}
\label{l2.2}
For all $(\alpha,(D,E))$ in $\rb^{n} \times
\ec$, $d_{\alpha,(D,E)}$ 
is a pseudo-metric
on $\uc\dc$.
The supremum
$\displaystyle
d := \sup_{
\stackrel{\alpha \in \rb^{n}}{(D,E) \in \ec}}
d_{\alpha,(D,E)}$ is a metric on $\uc\dc$,
valued in $[0, 1]$.
\end{lemma}

\begin{proof}
See Muraz and Verger-Gaugry \cite{murazvergergaugry2}.
\end{proof}

Let us show that ~$d$~ is invariant
by the action of the orthogonal group
$O(n, \rb)$.

\begin{lemma}
\label{rotation}
For all ~$(D, E) \in \ec, \alpha \in \rb^{n},
\rho \in O(n, \rb)$~ 
and 
~$\Lambda, \Lambda' \in \uc\dc$, the
following equality holds:
$$d_{\alpha, (D,E)}
(\Lambda, \Lambda')=
d_{\rho(\alpha), (\rho(D), E)}
(\rho(\Lambda), \rho(\Lambda')).$$
\end{lemma}

\begin{proof}
Let $(D,E) \in \ec$ and
$\bc_{m} \in \bc_{(D,E)}$
with $$\bc_{m} =
\{\stob(c_{1}^{(m)},
\epsilon_{1}^{(m)}),
\stob(c_{2}^{(m)},
\epsilon_{2}^{(m)}),
\ldots,
\stob(c_{i_{m}}^{(m)},
\epsilon_{i_{m}}^{(m)})\}.$$
The following inequalities
hold:
$$\| c_{q}^{(m)} - c_{k}^{(m)} \| \geq 1
\quad ~\mbox{for all}~ 
1 \leq q, k \leq i_{m}~ \mbox{with}
q \neq k.$$
Let $\rho \in O(n, \rb)$.
The collection $\bc_{m}$
is in one-to-one correspondence
with
the collection of open balls
$$\bc_{m}^{(\rho)} :=
\{\stob(\rho(c_{1}^{(m)}),
\epsilon_{1}^{(m)}),
\stob(\rho(c_{2}^{(m)}),
\epsilon_{2}^{(m)}),
\ldots,
\stob(\rho(c_{i_{m}}^{(m)}),
\epsilon_{i_{m}}^{(m)})\} \in
\bc_{(\rho(D),E)},$$ 
where the
following inequalities 
$$\| \rho(c_{q}^{(m)}) - \rho(c_{k}^{(m)}) \| \geq 1$$
are still true
for all
$1 \leq q, k \leq i_{m}$
with
$q \neq k$.
By \eqref{sondeii} 
the following equalities hold:
$$\phi_{\bc_{m}}
(\Lambda) = \phi_{\bc_{m}^{(\rho)}}
(\rho(\Lambda)).$$
Hence, for a given ~$\alpha \in \rb^{n}$, 
by
taking the supremum over all the
collections 
~$\bc_{m} \in \bc_{(D,E)}$~ 
of the following identity:
\\
\\
$\displaystyle
\frac{\left|
\phi_{\bc_{m}}
(\Lambda) -
\phi_{\bc_{m}}
(\Lambda')
\right|}{
\frac{1}{2} + \|\alpha\| +
\|\alpha - c_{1}^{(m)}\|+ \ldots +
\|\alpha - c_{i_{m}}^{(m)}\|
} =$\\
$\displaystyle
\mbox{ } \hspace{3cm}
\frac{\left|
\phi_{\bc_{m}^{(\rho)}}
(\rho(\Lambda)) -
\phi_{\bc_{m}^{(\rho)}}
(\rho(\Lambda')\right|
}{
\frac{1}{2}+
\|\rho(\alpha)\| +
\|\rho(\alpha) - \rho(c_{1}^{(m)})\|
+ \ldots +
\|\rho(\alpha) - \rho(c_{i_{m}}^{(m)})\|
}
$\\
\\
we deduce
the claim.
\end{proof}

By taking now the supremum of
$d_{\alpha, (D,E)}
(\Lambda, \Lambda')$
over all ~$\alpha \in \rb^{n}$~
and ~$(D,E) \in \ec$ we deduce from
Lemma \ref{rotation}
that
$$d(\Lambda, \Lambda') =
d(\rho(\Lambda), \rho(\Lambda'))$$
for all $\Lambda, \Lambda' \in \uc\dc$
and ~$\rho \in O(n, \rb)$
as claimed.

\vspace{0.5cm}

\section{Proof of Theorem \ref{metricuniv}}
\label{S6}

The metric ~$d$~ on
~$\uc\dc$~ (Theorem \ref{UDselection}) 
has the advantage 
to make compact the metric space 
~$(\uc\dc, d)$~ but, by
the way it is constructed,
the disadvantage to use a
base point (the origin) in the ambient 
space $\rb^{n}$. 
We now remove this disadvantage but
the counterpart is that 
the precompactness 
of the metric space
~$\uc\dc$ will be lost.
In order to do this, let us first
define
the new collection of metrics ~$(d_{x})$~
on ~$\uc\dc$ indexed by $x \in \rb^n$ by 
$$d_{x}
(\Lambda, \Lambda') =
d(\Lambda - x, \Lambda' - x), 
\hspace{1cm} \Lambda, \Lambda' \in \uc\dc.$$ 
Let us remark
that the metric spaces
$(\uc\dc, d_{x}), x \in \rb^{n},$
are all compact 
(by Theorem \ref{UDselection}).

\begin{definition}
\label{defiD}
Let $D$ be the metric on $\uc\dc$,
valued in $[0, 1]$, 
defined by
$$D(\Lambda, \Lambda') :=
\sup_{x \in \rb^{n}} \,
d_{x}(\Lambda, \Lambda')\, ,
\qquad \mbox{for} ~~\Lambda, \Lambda' \in \uc\dc.$$
The metric ~$D$~ is called 
{\it the metric of the
proximity of points}, or {\it pp-metric}. 
\end{definition}

{\it Proof of i)}:  
By construction, 
$D$ is invariant by the translations of
$\rb^{n}$. Let us prove 
its invariance
by the orthogonal group $O(n, \rb)$.
Let $\Lambda, \Lambda' \in
\uc\dc$~ and 
$x \in \rb^{n}, \rho \in
O(n, \rb)$. Since
$$d(\Lambda, \Lambda') =
d(\rho(\Lambda), \rho(\Lambda'))$$
by Lemma \ref{rotation}, we
deduce 
$$d_{x}(\Lambda, \Lambda') =
d(\Lambda - x, \Lambda' - x)=
d(\rho(\Lambda) - \rho(x), \rho(\Lambda') - \rho(x))
= d_{\rho(x)}(\rho(\Lambda), \rho(\Lambda')).$$
Hence,
$$\sup_{x \in \rb^{n}} d_{x}(\Lambda, \Lambda')=
\sup_{x \in \rb^{n}} d_{\rho(x)}(\rho(\Lambda), \rho(\Lambda')).$$
This implies 
$$D(\Lambda, \Lambda') =
D(\rho(\Lambda), \rho(\Lambda')).$$

{\it Proof of ii)}:
any Cauchy sequence for
the pp-metric $D$ is in particular
a Cauchy sequence
for the metric $d_{x}$
for all ~$x \in \qb^{n}$.
But $\qb^{n}$ is countable. 
Therefore, 
from any Cauchy sequence for $D$, 
a subsequence which converges 
for all the metrics ~$d_{x}, x 
\in \qb^{n}$, can be extracted
by a diagonalisation process over all
$x \in \qb^{n}$.
Since ~$\qb^{n}$~ is dense in 
$\rb^{n}$, that 
$$\sup_{x \in \qb^{n}} \,
d_{x}(\Lambda, \Lambda') = \sup_{x \in \rb^{n}} \,
d_{x}(\Lambda, \Lambda') \qquad 
\mbox{for all}
~\Lambda, \Lambda' \in \uc\dc_{r}$$
this subsequence, extracted by diagonalization, 
also converges for the
metric ~$D$. This prove the completeness
of the metric space 
$(\uc\dc, D)$.

{\it Proof of iii)}: 
we will use the pointwise pairing 
property of the metrics $d_x$ recalled
in the following Lemma.

\begin{lemma}
\label{pairingr1}
Let $x \in \rb^{n}$. 
Let 
~$\Lambda,
\Lambda' \in \uc\dc$~
assumed non-empty
and define
~$l_{x} := \inf_{\lambda \in \Lambda} \|\lambda - x\| 
< +\infty$. Let
$\epsilon \in (0, \frac{1}{1 + 2 l_{x}})$~
and let us assume that 
$d_{x}(\Lambda, \Lambda') <
\epsilon$. 
Then, for all
$\lambda \in \Lambda$ such that
$\|\lambda - x\| < \frac{1-\epsilon}{2 \epsilon}$,
\begin{itemize}
\item[(i)] there exists a unique
$\lambda' \in \Lambda'$ such that
$\|\lambda' - \lambda\| < 1/2$,
\item[(ii)] this pairing satisfies the inequality
$\|\lambda' - \lambda\| \leq 
(1/2+ \| \lambda - x\|) \epsilon$.
\end{itemize}
\end{lemma}

\begin{proof}
See Proposition 3.6 in \cite{murazvergergaugry2}.
\end{proof}

Let ~$0 < \epsilon < 1$~ and
suppose that 
$\Lambda, \Lambda' \in \uc\dc$
are non-empty
and satisfy
$D(\Lambda, \Lambda') < \epsilon$.
This implies
%that for all
%$x \in \rb^{n}$, in particular
%for all $\lambda \in \Lambda$,
%the inequality 
$$d_{\lambda}(\Lambda, \Lambda')
< \epsilon \qquad \mbox{for all}~ 
\lambda \in \Lambda.$$ 
From Lemma
\ref{pairingr1}, restricting
$x$ to all the elements
$\lambda$ of $\Lambda$,
we deduce 
$$\forall \lambda \in \Lambda, \, 
\exists
\lambda' \in \Lambda' 
~\mbox{(unique) such that}~ 
\|\lambda - \lambda'\| < \epsilon/2.$$
This proves the existence of unique pointwise 
pairings of points 
and the pointwise 
pairing property for $D$.

{\it Proof of iv)}:
let us show that
$$\uc\dc \times\rb^n ~~\to~~ \uc\dc$$
$$(\Lambda, t) ~~\to~~ \Lambda + t$$
is continuous.
Let $\Lambda_{0} \in \uc\dc$ and
$t_0 \in \rb^n$. First, by
the pointwise pairing property
given by iii), we deduce
$$\lim_{t \to 0}
D(\Lambda_{0} + t, \Lambda_{0}) =
0.$$ 
Let $0 < \epsilon < 1$.
Then, there exists $\eta > 0$ such that
$$|t-t_{0}| < \eta~ \Longrightarrow
~D(\Lambda_{0} + (t-t_{0}), \Lambda_{0})
< \epsilon/2.$$
Hence, for all 
$\Lambda \in \uc\dc$ such that
$D(\Lambda, \Lambda_{0}) < \epsilon/2$
and $t \in \rb^{n}$ such that
$|t-t_{0}| < \eta$, we have:
$$D(\Lambda + t, \Lambda_{0}+t_{0})
= D(\Lambda + (t-t_{0}), \Lambda_{0})$$
$$
\leq D(\Lambda + (t-t_{0}), \Lambda_{0} + (t-t_{0}))
+ D(\Lambda_{0} + (t-t_{0}), \Lambda_{0})$$
$$
= D(\Lambda, \Lambda_{0}) +
D(\Lambda_{0} + (t-t_{0}), \Lambda_{0})
\leq \epsilon/2 + \epsilon/2 = \epsilon.$$
We deduce the claim.

{\it Proof of (ii) (continuation)}:
let us prove that $(\uc\dc, D)$ is
locally compact. 
%Recall that t
The Hausdorff metric
$\Delta$
is defined on the set 
$\fc(\rb^{n})$ of
the non-empty closed subsets    
of $\rb^{n}$ as follows:
$$\Delta(\Lambda, \Lambda') :=
\max
\left\{
\, \inf
\{
\epsilon \mid
\Lambda' \subset
\Lambda + B(0,\epsilon)
\},
\, \inf
\{
\epsilon \mid
\Lambda \subset
\Lambda' + B(0,\epsilon)
\}
\right\}$$
in particular 
for $\Lambda, \Lambda' \in 
\uc\dc \setminus \{\emptyset\}$.
$\uc\dc \setminus \{\emptyset\}$ is closed in
the complete space 
$(\fc(\rb^{n}), \Delta)$. Then
$\uc\dc \setminus \{\emptyset\}$
is complete for $\Delta$.
On the space
$\uc\dc \setminus \{\emptyset\}$, the two metrics
$D$ and $\Delta$ are
equivalent.
The element
$\emptyset$ (system of spheres with no sphere)
is isolated in
$\uc\dc$ for $D$.
Hence, it possesses a neighbourhood
(reduced to itself) 
whose closure is compact.
Now, if $\Lambda
\in \uc\dc \setminus \{\emptyset\}$~
and $0 < \epsilon < 1$,
the open neighbourhood 
$\{\Lambda' \in \uc\dc \mid
\Lambda' \subset \Lambda + \stob(0, \epsilon)
\}$ of $\Lambda$
admits $\{\Lambda' \in \uc\dc \mid
\Lambda' \subset \Lambda + B(0, \epsilon)
\}$~ as closure which is obviously precompact,
hence compact,
for 
$D$ or $\Delta$. 
We deduce the claim. 
\vspace{0.5cm}

\section{Proofs of Theorem \ref{maxsphere} and Theorem
\ref{complsat}}
\label{S7}

Assume that there does not exist
$\Lambda \in \uc\dc$
such that 
\eqref{maximu} holds.
Then, by definition, there exists a sequence 
$(\Lambda_{i})_{i \geq 1}$ such that
$\Lambda_{i} \in \uc\dc$
%, m(\Lambda_{i}) = 1
and 
$$\lim_{i \to +\infty} \delta(\bc(\Lambda_{i})) 
= \delta_n$$
(as a sequence of real numbers).
\begin{lemma}
\label{subseq}
There exists a subsequence
$(\Lambda_{i_j})_{j \geq 1}$
of the sequence $(\Lambda_{i})_{i \geq 1}$ 
which converges for ~$D$.
\end{lemma}

\begin{proof}
Indeed, the sequence 
$(\Lambda_{i})_{i \geq 1}$
may be viewed as 
a sequence in the compact space
$(\uc\dc, d_{x})$ for any
$x \in \qb^{n}$. 
Therefore, for all
$x \in \qb^{n}$, we can extract a 
subsequence from it which converges
for the metric $d_{x}$. 
Iterating this extraction by a diagonalization
process over all $x \in \qb^{n}$, since
$\qb^{n}$ is countable, shows that we
obtain a subsequence which converges for
all the metrics $d_{x}$. Since
$\qb^{n}$ is dense in $\rb^{n}$, we obtain
a convergent sequence $(\Lambda_{i_{j}})_{j \geq 1}$
for $D$ since
$$\sup_{x \in \rb^{n}} d_{x}
= \sup_{x \in \qb^{n}} d_{x}.$$
\end{proof}

\begin{theorem}
\label{extract}
The density function
$\Lambda \to \delta(\bc(\Lambda)) = 
\|\chi(\bc(\Lambda))\|_1$
is continuous on 
$(\uc\dc, D)$ and
locally constant. 
\end{theorem}

\begin{proof}
Let  $\Lambda_{0} \in
\uc\dc, ~T > 0$~ large
enough
and $0 < \epsilon < 1$. 
By Lemma
\ref{pairingr1}
%, 
%\ref{pairing0r1} 
and
the pointwise pairing property
Theorem \ref{metricuniv} iii), 
any 
~$\Lambda \in \uc\dc$~ such that
~$D(\Lambda, \Lambda_{0}) 
< \epsilon$~ is such that
the number of elements
$\#\{
\lambda \in \Lambda ~|~
\lambda \in B(0,T)\}$~
of ~$\Lambda$~ within ~$B(0,T)$~ 
satisfies the following inequalities:
$$
\#\{
\lambda \in \Lambda_{0} \mid
\lambda \in B(0,T-
\epsilon/2)\}
\leq
\#\{
\lambda \in \Lambda \mid
\lambda \in B(0,T)\}
$$
$$\leq
\#\{
\lambda \in \Lambda_{0} ~|~
\lambda \in B(0,T+
\epsilon/2)\}.
$$
The density of the system of balls
~$\bc(\Lambda)$~ is equal to
$$\delta(\bc(\Lambda)) =
\limsup_{T \to +\infty}
\#\{
\lambda \in \Lambda ~|~
\lambda \in B(0,T)\} 
\left( \frac{1}{2 T}
\right)^{n}.$$
Since the contribution - to the 
calculation of the density -
of the
points of ~$\Lambda_{0}$~ 
which lie in the annulus 
~$B(0,T+
\epsilon/2)
\setminus
B(0,T-
\epsilon/2)$~
tends to zero
when
~$T$~ tends to infinity
by Theorem 1.8
in Rogers \cite{rogers}, we deduce 
that 
$$\delta(\bc(\Lambda)) 
= \delta(\bc(\Lambda_{0})),$$ 
hence the claim.
\end{proof}

Let us now finish the proof of Theorem 
\ref{maxsphere}.
Since the metric space
~$(\uc\dc, D)$~
is complete by 
Theorem \ref{metricuniv},
the subsequence
$(\Lambda_{i_{j}})_{j \geq 1}$ given
by Lemma \ref{subseq} is
such that  
there exists a limit point set
$$\Lambda  =
\lim_{j \to +\infty} \Lambda_{i_{j}} 
\in \uc\dc$$ 
which satisfies, 
by Theorem \ref{extract},
$$\delta_n = \lim_{j \to +\infty}
\delta(\bc(\Lambda_{i_{j}})) =
\delta(\bc(\Lambda)).$$
Contradiction.

Let us remark that,
in
this proof,
we did not need assume that
the elements
$\Lambda_{i_{j}}$ are saturated
(the same Remark holds for $m$-saturation). 

Let us prove Theorem
\ref{complsat}.
From Theorem 
\ref{maxsphere} 
there exists at least one
element of $\uc\dc$, say
$\Lambda$, of density
the packing constant $\delta_n$.
Let us assume that
there is no completely saturated
packing of equal balls of density
$\delta_n$
and let us show the contradiction. 
In particular we assume that
$\Lambda$ is not completetly saturated.

Then there would exist 
an application
$i \to m_{i}$ 
from $\nb \setminus \{0\}$ to
$\nb \setminus \{0\}$ 
and a non-stationary sequence 
$(\Lambda_{i})_{i \geq 1}$
such that
\begin{itemize}
\item[(i)] $\Lambda_{i} \in \uc\dc$ with
$\Lambda_{1} = \Lambda$,
\item[(ii)] ~$\Lambda_{i+1}$ is obtained from
$\Lambda_{i}$~ by removing $m_{i}$
balls and placing $m_{i}+1$ balls
in the holes formed by this removal process,
\item[(iii)] $\delta(\bc(\Lambda_{i})) = \delta_n$
for all $i \geq 1.$
\end{itemize}
This corresponds to a constant adding of new balls
by (ii), but since the density of
$\bc(\Lambda_1)$ is already maximal, equal to
$\delta_n$, this process ocurs
at constant density (iii).

As in the proof of Theorem 
\ref{maxsphere}, we can extract from the sequence
$(\Lambda_{i})_{i \geq 1}$
a subsequence $(\Lambda_{i_{j}})_{j \geq 1}$
which is a Cauchy sequence for $D$.
Since $(\uc\dc, D)$ is complete, there exists
$\Lambda \in \uc\dc$ such that
$$\Lambda = \lim_{j \to +\infty}\,
\Lambda_{i_{j}}.$$
The contradiction comes from
the pointwise pairing property
(iii) in Theorem \ref{metricuniv}
and the continuity
of the density function 
(Theorem \ref{extract}).
Indeed, 
for all $j$ large enough,
$D(\Lambda_{i_{j}}, \Lambda)$ is 
sufficiently small so that
the pointwise pairing property for $D$
prevents the adding of new balls to
$\Lambda_{i_{j}}$ whatever their number
by the process (ii).
Therefore, the subsequence  
$(\Lambda_{i_{j}})_{j \geq 1}$
would be stationary, which is excluded by assumption.  
This gives the claim.

%newpage (au lieu de \vfill\eject)

\vskip2cm
\begin{flushleft}
Gilbert MURAZ\\
INSTITUT FOURIER\\
Laboratoire de Math\'ematiques\\
UMR5582 (UJF-CNRS)\\
BP 74\\
38402 St MARTIN D'H\`ERES Cedex (France)\\
\smallskip
{\tt 
Gilbert.Muraz@ujf-grenoble.fr}

Jean-Louis VERGER-GAUGRY\\
INSTITUT FOURIER\\
Laboratoire de Math\'ematiques\\
UMR5582 (UJF-CNRS)\\
BP 74\\
38402 St MARTIN D'H\`ERES Cedex (France)\\
\smallskip
{\tt jlverger@ujf-grenoble.fr}
\end{flushleft}

\end{document}